\documentclass[12pt]{amsart}

\usepackage{bbm}

\usepackage{amsfonts,
            amsmath, 
            amssymb, 
            mathrsfs, 
            amscd, 
            verbatim,
            graphicx,
            color,
            wrapfig,
            tikz} 

\usetikzlibrary{matrix,arrows}

  \DeclareMathAlphabet{\mathpzc}{OT1}{pzc}{m}{it}

\newcommand{\onem}{\mathbbm{1}}
                   
\usepackage[font=small]{caption}
                   
\usepackage{psfrag}

\usepackage[T1]{fontenc} 

\input{epsf.sty}

\DeclareMathAlphabet{\mathpzc}{OT1}{pzc}{m}{n}

\def\Tgl{{\mathcal T}\!gl}

\renewcommand{\thefootnote}{$\fnsymbol{footnote}$}
\newtheorem{theorem}{Theorem}
\newtheorem{prop}[theorem]{Proposition}
\newtheorem{corr}[theorem]{Corollary}
\newtheorem{conj}[theorem]{Conjecture}

\newtheorem{defi}[theorem]{Definition}
\newtheorem{qes}[theorem]{Question}
\newtheorem{exe}[theorem]{Example}
\newtheorem{lemma}[theorem]{Lemma}

\def\blm{\begin{lemma}}
\def\elm{\end{lemma}}
\def\bdf{\begin{defi}}
\def\edf{\end{defi}}
\def\btm{\begin{theorem}}
\def\etm{\end{theorem}}
\def\bex{\begin{exe}}
\def\eex{\end{exe}}
\def\bpp{\begin{propos}}
\def\epp{\end{propos}}
\def\bq{\begin{qes}}
\def\eq{\end{qes}}

\def\ben{\begin{enumerate}}
\def\een{\end{enumerate}}

\def\eeq{\end{equation}}

\def\bda{\begin{displaymath}
\begin{array}{ccc}}
\def\eda{\end{array}\end{displaymath}}

\newcommand{\chapter}[1]{{\bf\Large #1}}

\newcommand{\txtqt}[2]
{\mbox{\raise .15em\hbox{$#1$}\!\big/\!\raise -.18em\hbox{$#2$}}}

\def\Rh{\mathbb Z[t^{\frac 12},t^{-\frac 12}]}

\def\QRh{\mathbb Q(t^{\frac 12},t^{-\frac 12})}

\newcommand{\str}[1]{\mathrm{Str}(#1)}

\newcommand{\pstr}[1]{\mathrm{pStr}(#1)}

\newcommand{\OhtsFct}{\mathcal V} 
\newcommand{\OFstr}[2]{\mathcal W_{#1,#2}}
\newcommand{\pOFstr}[2]{p\mathcal W_{#1,#2}}
\newcommand{\Burep}[1]{ {\mathcal B}_{#1}}

\newcommand{\LTWstr}[1]{ {\mathcal R}_{#1}}

\newcommand{\posar}[1]{\boldsymbol{\iota}^{#1}}

\newcommand{\spo}[1]{\overline{#1}}

\newcommand{\ndgop}{\boldsymbol{\theta}}
\newcommand{\dgop}{\theta}

\newcommand{\qtr}{\mathrm{{T\!r}}}

\newcommand{\Burmod}{M}

\newcommand{\lextalg}[1]{{\raise .4ex \hbox{\large $\bigwedge^{\mkern -1mu #1}$}}}
\newcommand{\extalg}[1]{{\raise .33ex \hbox{ $\mkern -4mu \bigwedge^{\mkern -5mu #1}\mkern -1mu$}}}
\newcommand{\sextalg}[1]{{\raise .33ex \hbox{\footnotesize $\bigwedge^{\mkern -1mu #1}$}}}

\newcommand{\sutr}{\mathrm{str}}
\newcommand{\STR}{\mathrm{Str}}

\newcommand{\BRT}[2]{\mathcal{Y}_{#1, #2}}

\newcommand{\Isom}{{\mathscr I}}
\newcommand{\wri}{\mbox{\large $\mathpzc w\mkern -2.6mu$}}
\newcommand{\swri}{\mbox{\small $\mathpzc w\mkern -2.6mu$}}

\newcommand{\eigv}{\mathbf v}
\newcommand{\eige}{\mathbf e}

\newcommand{\volel}[1]{\boldsymbol{\tau}^{#1}}

\newcommand{\covol}[1]{\raise 0.4ex \hbox{$\boldsymbol{\chi}$}_{#1}}

\newcommand{\grV}[2]{W_{#1, #2}}

\newcommand{\subs}[1]{\mathscr P_{#1}}

\newcommand{\eqp}{\mathbf{p}}

\begin{document}

\title{
Random Walk Invariants of String Links from R-Matrices} 

\author{Thomas Kerler  \ \ and \ \  Yilong Wang}

\address{Department of Mathematics, The Ohio State University\newline\indent Columbus, OH 43210, USA}
\email{kerler@math.ohio-state.edu} 

\address{Department of Mathematics, The Ohio State University\newline\indent Columbus, OH 43210, USA}
\email{wang.3003@math.ohio-state.edu} 

\begin{center}

\begin{abstract} We show that the exterior powers of the 
matrix valued random walk invariant of string links, introduced by Lin, Tian, and Wang, 
are isomorphic to the graded components of the tangle functor associated 
to the Alexander Polynomial by Ohtsuki divided by the zero graded invariant of the functor.
Several resulting  properties of these representations of the string link monoids
are discussed.
\end{abstract}

\end{center}

\maketitle 
\let\thefootnote\relax



\section{Introduction, Definitions, and Results}
 
\smallskip

\subsection{The Burau Representation and Tangles}

The aim of this article is to relate two generalizations of the 
Burau representation of the braid groups to certain types of tangles. Consider 
the standard presentation of the braid group in $n$ strands   
\begin{equation}\label{eq-BnPres}
B_n=\left\langle \sigma_i,\,i=1\ldots n-1{\Bigl |}\begin{array}{c}
\sigma_i\sigma_{i + 1}\sigma_i = \sigma_{i +1}\sigma_i\sigma_{i +1};\ \  i=1\ldots n-2\\
\sigma_i\sigma_j = \sigma_j\sigma_i; \ \ |i - j| \geq 2\;\\
\end{array}
\right\rangle\;.
\end{equation}

The {\em unreduced} Burau representation $B_n$ is defined on the free 
$\mathbb{Z}[t,t^{-1}]$-module of rank $n$, given by the homomorphism
\begin{equation}\label{eq-burauR}
\Burep n :\,B_n\rightarrow \mathrm{End}(\mathbb{Z}[t,t^{-1}]^n):\,\sigma_i\mapsto\Burep n (\sigma_i)=\beta_i\,.
\end{equation}

Here the Burau matrices $\beta_i$ and their inverses, denoting $\overline t=t^{-1}$, are defined as  
\begin{equation}\label{eq-burauM}
\beta_i=\onem_{i - 1}\oplus \left[\begin{array}{cc}
(1 - t)&1\\
t&0\\
\end{array}\right]\oplus \onem_{n - i - 1}
\quad \mbox{,}\quad 
\beta_i^{-1}=\onem_{i - 1}\oplus \left[\begin{array}{cc}
0 & \overline t\\
1&(1 - \overline t)\\ 
\end{array}\right]\oplus \onem_{n - i - 1}\,.
\end{equation}

\begin{wrapfigure}{R}{45mm} 
\begin{center}
\vspace*{-5mm}
\psfrag{n1}{\small $2$}   
\psfrag{n2}{\small $1$}   
\includegraphics[width=37mm]{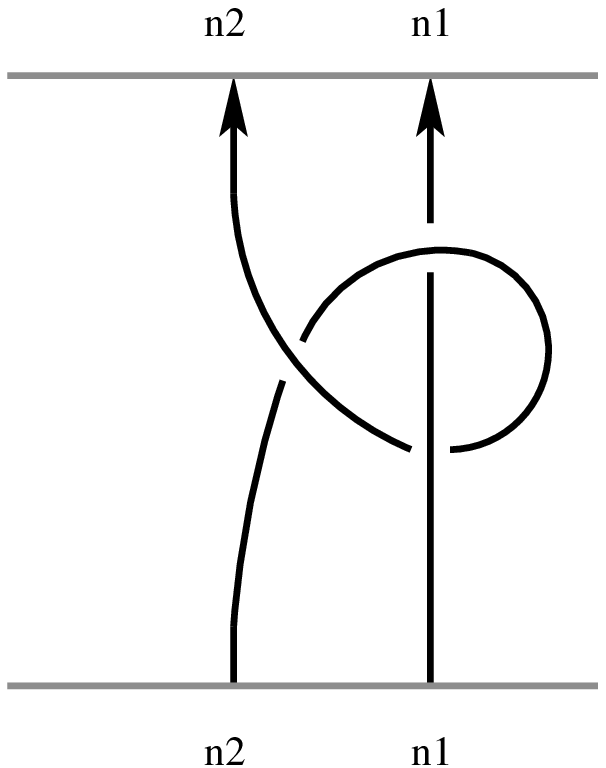}
\vspace*{1mm}
\caption{{\small $S:\posar 2\to \posar 2$}}
\label{fig:ex_strlk}
\vspace*{-3mm}
\end{center}
\end{wrapfigure}
There are various generalizations of the Burau representation from braids to tangles. 
To explain these recall first  the 
usual category of oriented tangles, denoted by $\Tgl$, whose objects are tuples of signs 
$\underline\epsilon=(\epsilon_1,\ldots,\epsilon_n)$ writing $|\underline\epsilon|=n$.
 A morphisms $T:\underline\epsilon\to \underline\delta$ is
an equivalence classes of oriented tangle diagrams in $\mathbb R\times[0,1]$ with endpoints 
$\{(j,0), j=1\ldots |\underline\epsilon|\}$
at the bottom of the diagram and $\{(j,1), j=1\ldots |\underline\delta|\}$ at the top
so that the orientation is upwards at the $j$-th positions if $\epsilon_j=+$ and downwards
if $\epsilon_j=-\,$. The equivalences are given
by isotopies and the usual Reidemeister moves.

We denote by $\posar n=(+,\ldots,+)$ the array of length $n$ with all + entries, which
implies that the orientations of all strands at this objects (either as source or target) 
must be pointing upwards.

A string link is a tangle class $T:\posar n\to \posar n$ with precisely $n$ interval components each of which 
has one endpoint $(i,0)$ at the bottom and the other endpoint $(j,1)$ at the top of the diagram.
See Figure~\ref{fig:ex_strlk} for an example of a string link $S:\posar 2\to \posar 2$ on two strands. 
String links  form a monoid $\str n$ with respect to the composition in $\Tgl$ so that  
we have the following inclusions of monoids:
\begin{equation}\label{eq-incls}
 B_n\;\subset\;\str n \;\subset\; \mathrm{End}_{\Tgl}(\posar n)\,.
\end{equation}

\subsection{Random Walk on String Link Diagrams}\label{s-randomw}

In \cite{LTW98} Lin, Tian, and Wang consider a  generalization of $\Burep n $ 
to $\str n$ which is inspired by a remark Vaughan Jones in his seminal paper \cite{Jo87},
where he offers a probabilistic interpretation of the Burau representation.
The description there is in terms of a bowling
ball that runs along an arc in a braid diagram following its orientation. At positive crossings 
the ball drops from an overcrossing strand to an undercrossing 
strand with  probability $(1-t)$ and remains on the overcrossing strand with probability $t$.

\begin{wrapfigure}{R}{38mm}
\begin{center}
\vspace*{-6mm}
\psfrag{x}{\small $i$}  
\psfrag{y}{ \small $(i+1)$}  
\psfrag{a}{\small $(1-t)$}  
\psfrag{b}{\small $t$}  
\includegraphics[width=27mm]{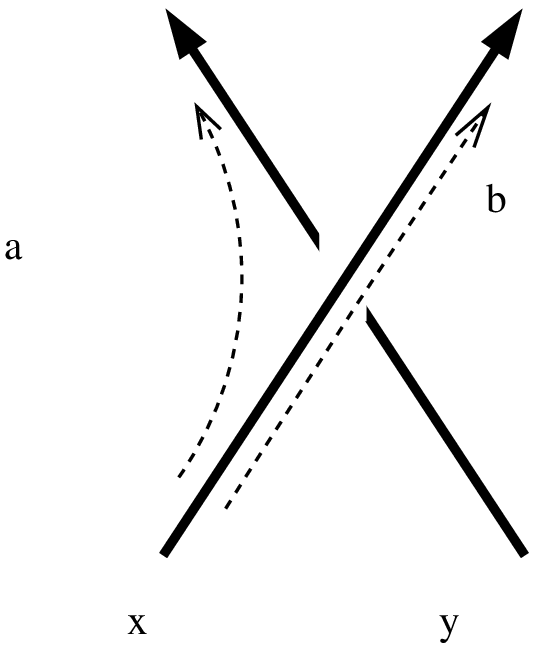}
\vspace*{1mm}
\caption{}\label{fig:cross_jump}
\vspace*{-4mm}
\end{center}
\end{wrapfigure}
The situation is illustrated in Figure~\ref{fig:cross_jump}. An analogous rule is used for 
negative crossings in which $t$ is replaced by $\overline t=t^{-1}$, see Figure~\ref{fig:X_opp_jump}, 
so that ``probabilities''
should rather be understood as weights whose values are allowed to be outside of the unit interval. 
Assuming the pictured strands are the  $i$-th and $(i+1)$-st 
strands in a braid presentation on $n$ strands the Markov transition matrix from the array of
probabilities just above the crossing to one just below is given by the unreduced
Burau matrices in (\ref{eq-burauM}). 

The construction in \cite{LTW98} extends the idea of weighted paths along an oriented
string link diagram with the same rule as in Figure~\ref{fig:cross_jump}. As opposed
to braid diagrams it is possible to encounter loops and thus infinite numbers of 
paths between two endpoints. It is shown that for $t$ sufficiently close to 1 the
resulting series of weights converge. For example,
the diagram in Figure~\ref{fig:ex_strlk} contains a loop of combined
weight $w=t(1-\overline t)$ and associated geometric series $\sum_{k=0}^{\infty}w^k=(2-t)^{-1}\,$.
The combined transition matrix in this example is given as
\begin{equation}\label{eq-LTWstr-exS}
 \LTWstr 2 (S)\,=\,
 \frac 1 {2-t}\left[\begin{array}{cc}
                     1 &  \overline t -1\\
                     1-t & 3-t-\overline t
                    \end{array}
\right]\,.
\end{equation}

It is shown in \cite{LTW98} that the resulting transition matrix will always have
rational functions in $t$ as entries. The functorial nature of the construction,
see Lemma~\ref{lm-basicsLTW} below, further
implies a homomorphism of monoids
\begin{equation}\label{eq-LTWhom}
 \LTWstr n \,:\;\str n \,\longrightarrow \, \mathrm{End}(\mathbb{Q}(t,t^{-1})^n)\;,
\end{equation}
which restricts to $\Burep n $ on the braid group $B_n$. The main result of this paper
will also extend to the exterior powers of this representation 
\begin{equation}\label{eq-extLTWhom}
\begin{split}
 \extalg{k}\LTWstr n \,& :\;\str n \,\longrightarrow \, \mathrm{End}(\extalg{k}\mathbb{Q}(t,t^{-1})^n)\\
 \mbox{given by }\qquad & \extalg{k}\LTWstr n (T).(x_1\wedge\ldots\wedge x_k)=(\LTWstr n (T)x_1)\wedge\ldots\wedge (\LTWstr n (T)x_k)\;.
\end{split} 
\end{equation} 

The construction of the random walk invariant in \cite{LTW98} has been applied to studies of
the Jones polynomial and its relation to the Alexander polynomial in \cite{LW01,GL05} 
and has been generalized in \cite{CT05,KLW01,SW01}.

\subsection{Functorial Invariants of Tangles}\label{ss-functgl}
A second generalization to the entire category of oriented 
tangles is given in Section 3.3 of Ohtsuki's book \cite{OH02} as operator invariants
a tangles associated to the Alexander polynomial. 
The matrices assigned to generating tangles there  have entries in the ring 
of Laurent polynomials in $t^{\pm\frac 12}$
for suitable basis. The construction
can thus be summarizes as a functor
\begin{equation}\label{eq-OhtsFct}
 \OhtsFct:\,\Tgl \;\longrightarrow\; \mathbb Z[t^{\frac 12},t^{-\frac 12}]-{\rm Fmod}
\end{equation}

from the tangle category to the category of free $\mathbb Z[t^{\frac 12},t^{-\frac 12}]$-modules. 
The assignment on objects is given by 
$\OhtsFct(\underline{\epsilon})=V^{\epsilon_1}\otimes\ldots \otimes V^{\epsilon_n}$
where $V^+=V$ is the free $\mathbb Z[t^{\frac 12},t^{-\frac 12}]$ module of rank 2
generated by elements $e_0$ and $e_1$, and $V^-=V^*$ is the dual module with
dual basis $\{e_0^*,e_1^*\}$. In fact,
$\OhtsFct$ is easily seen to be a tensor functor with respect to the tensor product
on $\Tgl$ defined by the usual juxtaposition.

As described in Section 4.5 of \cite{OH02} the tangle functor may be derived from
the representation theory of the quantum group $U_{-1}(\mathfrak{sl}_2)$ following 
 the methods originally introduced by Reshetikhin and Turaev \cite{RT90}. 
One implication
of this observation, which can also be checked directly, is that for a tangle $T$ 
the operator $\OhtsFct(T)$ preserves a natural grading on 
the the modules induced, for example, by $\deg(e_0)=\deg(e_0^*)=0$, $\deg(e_1)=1$, and $\deg(e_1^*)=-1$.
For example, the module assigned to the object of length $n$ with all positive orientations
decomposes, as free $\mathbb Z[t^{\frac 12},t^{-\frac 12}]$-modules, 
into its invariant graded components as follows:
\begin{equation}\label{eq-OFstrDec}
 \OhtsFct(\posar n)=V^{\otimes n}=\bigoplus_{j=0}^n\grV n j\qquad \mbox{with} \ \quad \dim(\grV n j)= {n \choose j} \;.
\end{equation}

A consequence of this decomposition is that the representation of $\str n$ on $\OhtsFct(\posar n)$
implied by the inclusion in (\ref{eq-incls}) yields a series of representations as follows:
\begin{equation}\label{eq-wdefLTW}
 \OFstr n j:\, \str n\,\longrightarrow\,\mathrm{End}(\grV n j)\,:\;T\,\mapsto\,\OhtsFct(T)|_{\grV n j}\,.
 \end{equation}

We note that $\grV n 0$ is of rank one so that we may write 
$\OFstr n 0(T)\in \Rh$ as a well-defined 
polynomial-valued invariant. We will see in Lemma~\ref{lm-t=1OF} that its specialization at $t=1$
is one for all string links so that $\OFstr n 0(T)\neq 0$ also as an element in $\Rh$. 
We can therefore define representations of the string monoid as follows.
\begin{equation}\label{eq-def10quot}
\OFstr n {k/0}:\;\str n\to \mathrm{End}(\QRh^n)\,:\; T\,\mapsto \,\frac 1 {\OFstr n 0(T)}\OFstr n k(T)\;.
\end{equation} 

It is shown in \cite{OH02} that the restriction of 
$n$-dimensional representation 
$\OFstr n 1:\,B_n\to \mathrm{GL}(\grV n 1)$ to the braid group is, up to 
a universal rescaling of generators,
equivalent to the unreduced Burau representation $\Burep n $ from (\ref{eq-burauR}).

An closely related approach of constructing tangles invariants associated to the
Burau representation and Alexander polynomial makes use of the quantum groups 
$U_{\zeta}(\mathfrak{gl}(1|1))$, see for example \cite{KS91,M91,V07}.

\subsection{Statement of Main Result}\label{ss-statement}

In view of the dominance of algebraically constructed functorial invariants it is
natural to ask whether the representation $\LTWstr n $ is really a special case of these.
Among the obstacles in an identification is the peculiar analytic flavor of the 
construction involving geometric series. Particularly, the role of the denominators 
that are obtained by summing these series is, as for example $(2-t)$ in 
(\ref{eq-LTWstr-exS}), are not obvious and their  algebraic or topological meaning 
are not at all immediate from the construction.

The main result of this article is to provide such algebraic interpretations in terms of an
equivalence of string link representations  as stated in the following theorem.

\begin{theorem}\label{thm-main}  The representations 
$\extalg{k}\LTWstr n $ and $\OFstr n {k/0}$ of the string link monoid $\str{n}$ are
isomorphic to each other.
\end{theorem} 
 
 We note that the isomorphism will consist only of rescaling or reordering of basis vectors. 
 In the case of $k=1$ we obtain a direct formula for the string link invariant:
 \begin{equation}\label{eq-01spcase}
  \Isom_n^{-1}\LTWstr{n}(T)\Isom_n=\frac 1 {\OFstr{n}{0}(T)} \OFstr{n}{1}(T)\;,
 \end{equation}
 where $\Isom_n$ is given as a bijection of canonical basis vectors up to multiplication with 
 units of $\Rh$. 
The denominator occurring in the random walk construction can thus also be  interpreted 
as the zero-graded part of Ohtsuki's functor. Additional interpretations are suggested in
Section~\ref{ss-comm}. 
  
 An immediate implication of Theorem~\ref{thm-main} is that $\LTWstr n $ is dominated by finite
 type invariants since it is dominated by $\OhtsFct$. This fact has been proved by more indirect 
 means in  \cite{LTW98}. Another 
 consequence of Theorem~\ref{thm-main} is the following.
 \begin{corr}\label{corr-main} Suppose $T\in\str{n}$.
  
  Then $\OFstr{n}{0}(T)^{k-1}$ divides all $k\times k$ minors of $\OFstr{n}{1}(T)$ in $\Rh$.  
 \end{corr}

 \subsection{Overview of Paper}
 
The original random walk construction of $\LTWstr{n}$ \cite{LTW98} is reviewed and formalized in Section~\ref{s-rwi},
where it is also applied to the situation of a string link $T$ given as the closure of a braid $b$. Particularly, we obtain
in Proposition~\ref{prop-blfrm} an expression for  $\LTWstr{n}(T)$ in terms of blocks of the Burau matrix for $b$.
A consequence, stated in Corollary~\ref{corr-LTWeigen}, is that $\LTWstr{n}(T)$ admits an equilibrium state that is
independent of $T$ (and thus contains no information about $T$).

In Section~\ref{s-Rmat} we review Ohtsuki's construction of the tangle functor  $\OhtsFct$ as well as the
equivalence given in \cite{OH02} of the implied braid representation with the  exterior algebra 
extension of the unreduced Burau representation. The quantum trace, relevant to evaluating braid closures, is 
related to the natural supertrace on exterior algebras in Section~\ref{ss-TrRel}. In addition various 
grading and equivariance properties are discussed.

Section~\ref{ss-proof} contains the proofs for Theorem~\ref{thm-main} and Corollary~\ref{corr-main} after 
introducing several technical lemmas on relating traces, evaluations on top forms, and Schur complements of
block matrices in Sections~\ref{ss-strviatop} and \ref{ss-trschur}. Finally, we present additional points of view 
and possible further questions of study in Section~\ref{ss-comm}. 

\smallskip

\noindent
{\bf Acknowledgment:} The first author thanks Craig Jackson for discussions and calculations on an early 
version of the conjecture that are documented in \cite{J01}.


\section{Random Walk Invariants of Tangles} \label{s-rwi}

After a brief review of the random walk construction of \cite{LTW98} the main result of
this section is a formula for the representation $\LTWstr n $ in Proposition~\ref{prop-blfrm}  in terms
of block matrices of a Markov presentation of the string link. 

\subsection{Weighted Path Construction}\label{ss-xslin}

We will review here the construction of  \cite{LTW98}, formalizing the intuition in terms 
of random walks given in the introduction. A string link $T\in \str n$  is an oriented 
tangle $T:\posar n\to \posar n$ for which each component is an interval that 
starts at a bottom point $(i,0)$ and end at a top point
$(\pi_T(i),1)$ with $\pi_T\in S_n$.

An admissible path $P$ in a diagram of $T$ is a path following
the orientation of $T$ at each piece of the diagram. At a crossing the path must continue 
its direction on $T$ if it is along the undercrossing piece of the crossing. If the path 
approaches a the crossing along the overcrossing piece it may either continue in the same 
direction or continue on the undercrossing piece in the respective direction. In addition
an admissible path needs to start at a bottom point $(i,0)$ and end at a top point $(j,1)$.
Thus locally a path near a crossing may look like one of the dashed lines in 
Figure~\ref{fig:cross_jump} or \ref{fig:X_opp_jump}.

For braids all admissible paths need to travel upwards so that they will pass through 
each crossing at most once. However, for a string link, such as the one in Figure~\ref{fig:ex_strlk},
admissible paths may loop through a crossing arbitrarily often so that there are 
infinitely many admissible paths.

\begin{wrapfigure}{R}{35mm}
\begin{center}
\vspace*{-2mm}  
\psfrag{a}{\small $(1-\overline t)$}  
\psfrag{b}{\small $\overline t$}  
\includegraphics[width=22mm]{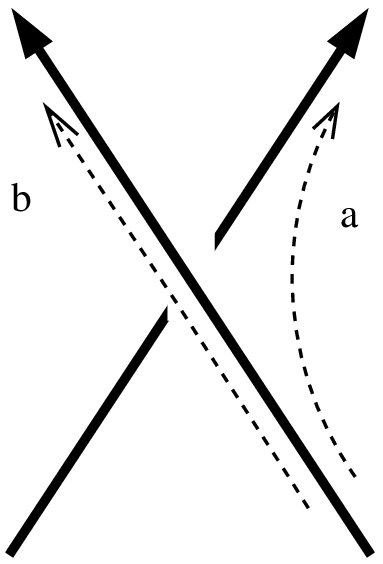}
\vspace*{1mm}
\caption{}\label{fig:X_opp_jump}
\vspace*{-2mm}
\end{center}
\end{wrapfigure}

To an admissible path $P$ that passes through $M$ crossings (counting repetitions) 
we associate a weight $w(P)=w_1\cdot\ldots\cdot w_M$, where $w_k=1$ if $P$ approaches
the $k$-th crossing along an undercrossing piece. If $P$ approaches the $k$-th crossing
along an overcrossing piece and the crossing is positive as in Figure~\ref{fig:cross_jump}
we set $w_k=t$ if $P$ is continuing in the same direction and $w_k=1-t$ if $P$ 
continues on the undercrossing piece. For a negative crossing as in Figure~\ref{fig:X_opp_jump}
we assign $w_k=\overline t=t^{-1}$ if $P$ is continuing in the same direction and 
$w_k=1-\overline t=1-t^{-1}$ otherwise.

We uses these to assign to the diagram of
a string link $T:\posar n\to \posar n$ an $n \times n$ matrix $\LTWstr n (T)$ 
whose $(i, j)$ entry is given by
\begin{equation}\label{eq-LTWij-paths}
\LTWstr n (T)_{j, i} = \displaystyle\sum_{P\in \mathcal P_i^{j}} w(P)\;\in\mathbb Q(t,t^{-1})\,.
\end{equation}
Here the summation is over the set $\mathcal P_i^{j}$ of
all admissible paths in a diagram of $T$ from the point $(i, 0)$ to $(j, 1)$.
The entry is 0 if there is no  such path. 

It is shown in \cite{LTW98} that for $t$ sufficiently close to 1 the summations (\ref{eq-LTWij-paths}) over all paths
converge indeed to rational expressions in $t$ and that these expressions are not dependent on the 
diagram chosen to present a particular $T$. 

\begin{lemma}\label{lm-basicsLTW} The assignment in (\ref{eq-LTWij-paths}) has the following basic properties:
 \begin{enumerate}
  \item The assignment $T\mapsto\LTWstr n (T)$ obeys $\LTWstr n (T)\LTWstr n (S)=\LTWstr n (T\circ S)$, where the composite
 $T\circ S$ in $\Tgl$ is given by stacking $T$ on top of $S$. 
 \item When restricted to the braid group $B_n$ the assignment reduces to the Burau representation as
 defined in (\ref{eq-burauR}).
 \item Specializing $\LTWstr n (T)$ to $t=1$ we  obtain the matrix of the permutation $\pi_T$ associated to $T$, that is,
 $\LTWstr n (T)_{j,i}=\delta_{j,\pi_T(i)}$.
 \end{enumerate}
 \end{lemma}

 \begin{proof} For the first part note that at the boundary between $T$ and $S$ in $T\circ S$ all orientations are
 upwards so that no admissible path can return from $T$ to $S$. Thus any admissible path $Q:i\to k$ in $T\circ S$ 
 is the composite of a path $P:i\to j$ in $S$ and a path $R:j\to k$ in $T$ for some $j$. Also the weight is clearly
 multiplicative $w(Q)=w(R)w(P)$. Summation over all $R$, $P$, and $j$ thus yields the respective matrix element
 for $\LTWstr n (T\circ S)$, which is thus the matrix product as desired.
 
 As described in the introduction the assignment coincides with the Burau representation on the generators so that
 by the previous the two repetitions coincide on all elements on $B_n$.
 
 For the third claim notice that all paths that change direction will have weight zero and may thus be discarded.
 For given indices the remaining path that run strictly 
 along a tangle component (preserving direction at every crossing) will have weight one. 
 \end{proof}

\subsection{Braid Closures and Burau Matrix Blocks}\label{ss-mbm}

Given a braid $b\in B_{n+m}$ thought of as an isomorphism on $\posar {n+m}$ 
we can construct a a tangle $T:\posar{n}\to\posar{n}\,$ by closing the last $m$ strands 
by loops as indicated in Figure~\ref{fig:Markov_Pres}. The following lemma
is a nearly straightforward generalization of Alexander's Theorem for links 
with some additional attention given to orientations at the end points.
A respective generalization of the Markov Theorem also holds but is not needed here
since we are only concerned with the comparison rather than the construction of invariants.
\begin{lemma}
 Every oriented tangle $T:\posar{n}\to\posar{n}\,$ is given as the partial closure of 
 a braid $b\in B_{n+m}$ for some $m$ (as in Figure~\ref{fig:Markov_Pres}).
\end{lemma}

\begin{figure} 
\begin{center}
\vspace*{-1mm}  
\psfrag{D}{\huge $\ldots$}  
\psfrag{M}{\small $m$}  
\psfrag{N}{\small $n$}  
\psfrag{b}{\large $b$}  
\psfrag{T}{\large $T=$}  
\psfrag{L}{$\mathcal L$}  
\psfrag{p}{\footnotesize $\mathcal O$}  
\includegraphics[width=78mm]{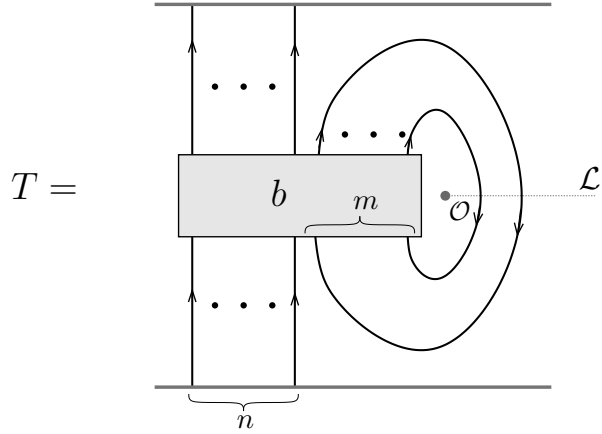}
\vspace*{-2mm}
\caption{Markov Presentation}\label{fig:Markov_Pres}
\vspace*{-1mm}
\end{center}
\end{figure}

\begin{proof}
 Most of the proof is nearly verbatim the same as, for example, the one 
 in Section~6.5 of \cite{PS97}.
 The diagram is assumed to be polygonal with vertical endpoints and a
 rotation point {\small $\mathcal O$} to the right of the diagram at mid height. 
 We may assume that no line segments are in radial direction from  {\small $\mathcal O$}
 and, using subdivision, that each segment has at most one crossing point.
 Thus each segment is either in clockwise or counter-clockwise direction around 
 {\small $\mathcal O$}. The {\em Alexander trick} is applied to every line segment
 in counter-clockwise direction as depicted in Figure~6.4 of  \cite{PS97}
 depending on the orientation of a possible crossing. As a result all segments will
 be in clockwise direction.
 
 Observe that the line segments at the top and bottom end points are already in
 clockwise direction so that they are not affected by the algorithm. In the 
 resulting tangle diagram we will have $m$ segments intersecting the horizontal
 line $\mathcal L$ indicated in Figure~\ref{fig:Markov_Pres}. Each of these
 segments can be arranged to be vertical and thus oriented downwards.  
 To these we apply the cut and stretch process described in Figure~6.3 of \cite{PS97}.
 As a result segments at all crossings are oriented upwards and to the left of
 {\small $\mathcal O$} so that we obtain the desired braid closure presentation.
\end{proof}

For a given tangle $T:\posar{n}\to\posar{n}$ that is the closure of a braid $b\in B_{n+m}$
in the above fashion consider the block form of the $(n+m)\times (n+m)$ Burau matrix 
associated to $b$ and its specialization at $t=1$.
\begin{equation}\label{eq-bclmat}
 \Burep n (b)=\left[\begin{array}{ccc}
X&&Y\\
&&\\
Z&&Q\\
\end{array}\right]
\qquad \mbox{ and } \qquad  
 \pi_b=\Burep n (b)|_{t=1}=\left[\begin{array}{ccc}
\spo{X}&&\spo{Y}\\
&&\\
\spo{Z}&&\spo{Q}\\
\end{array}\right]
\end{equation}
Here  $X$ is an $n\times n$-matrix and $Q$ is an $m\times m$-matrix, each with coefficients in $\mathbb Z[t,t^{-1}]$. 
The dimensions of the matrices $\spo{X}$, $\spo{Y}$, $\spo{Z}$, and  $\spo{Q}$ are the same but all matrix entries 
are either 0 or 1.

\begin{lemma}\label{lm-Qnil}
Suppose $T:\posar{n}\to\posar{n}$ is a tangle that is the closure of a braid $b\in B_{n+m}$ and
which has associated matrices as in (\ref{eq-bclmat}).

Then $T$ is a string link if and only if $\spo{Q}$ is nilpotent.
\end{lemma}

\begin{proof} We note first that, since $\spo{Q}$  is the block of a permutation matrix $\pi_b$, 
it has at most one 1 in each row and each column and 0's for all other entries. That is, it is the 
incidence matrix of an oriented graph $I_Q$ with each vertex having at most on incoming and one
outgoing edge. Clearly the components of such a graph are either  oriented intervals or oriented circles.
Thus, after reordering of the basis, the matrix can be brought into a block diagonal form 
with two types of blocks.
The first, corresponding to interval components of $I_Q$, are nilpotent Jordan blocks given by
square matrices $N$ with $N_{i,j}=\delta_{j,i+1}$. The second type for the circle components 
are cyclic $k\times k$-matrices $C$ with $C_{i,j}=1$ if $j\equiv i+1 \mod k$ and $C_{i,j}=0$
otherwise. Hence $\spo{Q}$ is nilpotent if an only if there are no cyclic blocks.

Consider first a matrix block of the first nilpotent type $N$. Denoting canonical basis 
$\{e_j:j=1,\ldots, n\}$ we have for some $p$ and $q$ with $n<p<q\leq n+m$ that
$\pi_be_s=Ne_s=e_{s-1}$ for $s=p+1,\ldots ,q$. Since $\pi_b$ encodes which points at the 
bottom of $b$ are connect by strand to which points at the top of $b$ 
we find  that the arcs at positions $p$ through $q$ are consecutively connected to each other
by intervals in $b$. They thus form one interval $J$ in $T$ which is starting at the $q$-th position 
at the top of $b$ and ending at the $p$-th position at the bottom of $b$ and intersecting
$\mathcal L$ in $q-p+1$ arcs. Since $N$ and thus $\spo{Q}$ has only 0 entries
in the $p$-th column we must have a 1 entry in the $p$-th column for $\spo{Y}$ in some 
$k$-th row (with $k\leq n$) which means $J$ is connected to the $k$-th top point of the diagram
for $T$. Similarly, since $N$ and thus $\spo{Q}$ has only 0 entries in the $q$-th row 
$\spo{Z}$ must have a 1 entry in the $q$-th row in some $l$-th column so that $J$ must be
connected to the $l$-th start point at the bottom of the tangle diagram.

Thus if all matrix blocks of $\spo{Q}$ are of nilpotent type the components of all closing
arcs are connected to top and end bottom points of the tangle diagram. Components of $T$
that are disjoint from closing arcs always oriented upwards and thus 
also must connect to top and bottom points of the diagram. Thus of $\spo{Q}$ is nilpotent
$T$ is indeed a string link.

Conversely, suppose $\spo{Q}$ contains a cyclic block with $\pi_be_s=Ce_s=e_{s-1}$ 
for $s=p+1,\ldots ,q$ and $Ce_p=e_q$. Then the arcs connected at positions $p$ through $q$
are again consecutively connected by intervals of $b$ into one component but now the ends 
of the component are also connected by an interval in $b$ forming a closed component of $T$.
However, closed components are not allowed for a string link. 
\end{proof}

\subsection{Block Matrix Formula for String Links}\label{ss-blf}

The following consequence of nilpotency
will be useful to control geometric series
occurring in the random walk picture.

\begin{lemma}\label{lm-Qest}
 Suppose $Q$ is an $m\times m$-matrix that depends continuously on a real (or complex) parameter $t$ in
 a vicinity of $t=1$ Assume also that the
 specialization $\spo{Q}$ at $t=1$ is nilpotent. 
 Then for any $d>0$ there exist a $\varepsilon>0$ and a $C>0$ such that for all $t$ with $|t-1|<\varepsilon$
 and all integers $N>0$ we have 
 \begin{equation}
  \|Q^N\| \leq C d^N\;.
 \end{equation}
\end{lemma}

\begin{proof} Assume that $Q$ is continuous for $t$ with $|t-1|\leq \varepsilon_1$ and let 
$M>d$ be an upper bound of $\|Q\|$ for these $t$. Now since $\spo{Q}^m=0$ by assumption we have 
that $\|Q^m\|$ is a continuous function vanishing at 0 so that there is an $\varepsilon\in (0,\varepsilon_1)$
such that $\|Q^m\|<d^m$ whenever $|t-1|<\varepsilon$. Writing $N=cm+r$ with $r=0,\ldots, m-1$ we thus 
have $\|Q^N\|=\|Q^r Q^{mc}\|\leq \|Q\|^r\|Q^m\|^c\leq M^rd^{mc}=(\frac M d )^rd^N\leq Cd^N$, where
$C=(\frac M d )^{m-1}$. 
\end{proof}

 In the proposition below we  establish a relation between the blocks of the Burau representation
 of a braid and the random walk invariant of the string link obtained by closing the same braid.

 \begin{prop} \label{prop-blfrm}
 Suppose $T \in \str{n}$ is a string link obtained as the closure of a braid $b\in B_{n+m}$
 and let $X$, $Y$, $Z$, and $Q$ be the matrix blocks of $\Burep n (b)$ as in (\ref{eq-bclmat}). Then $(\onem -Q)$ is
 invertible over $\mathbb Q(t,t^{-1})$ and we have 
 \begin{equation}\label{eq-blfrm}
  \LTWstr n (T) = X+Y(\onem -Q)^{-1}Z\;.
 \end{equation}
 \end{prop}

 \begin{proof}
  We note first that for $D(t)=\det(\onem -Q)$ we have $D(1)=1$ since $\spo{Q}$
  is nilpotent so that $D(t)\neq 0$ and  $(\onem -Q)$ is indeed
 invertible over $\mathbb Q(t,t^{-1})$. Thus the expression on the right side of (\ref{eq-blfrm})
 is always well defined.
 
 In order to evaluate $\LTWstr n (T)_{ji}$ we partition the set of admissible paths 
 $\mathcal P_i^{j}$ further. Note that every path $P\in \mathcal P_i^{j}$ is characterized by how
 often and in which order it will pass through the $m$ closing arcs attached to positions 
 $n+1$ through $n+m$ of $b$ as in Figure~\ref{fig:Markov_Pres}. Denote by $\mathcal P_i^{j}(i_1\ldots i_k)$
the set of admissible paths that pass through arcs in positions $i_1, \ldots, i_k\in\{n+1,\ldots,n+m\}$
in this order (and with repetitions allowed). Denote also by  $\mathcal M^{t}_s$ the set of admissible paths in $b$ that start at the
$s$-th position at the bottom of $b$ and end at the $t$-th position at the  top of $b$. 
Then it is clear that each path in the former is put together in a unique fashion by pieces from
the latter sets, yielding a natural bijection as follows.
 \begin{equation}
  \mathcal P_i^{j}(i_1\ldots i_k)\cong 
 \mathcal M_{i_{k}}^{j}
 \times \mathcal M_{i_{k-1}}^{i_k}
 \ldots
 \times \mathcal M_{i_1}^{i_2}
 \times \mathcal M_i^{i_1}\;.
 \end{equation}
 Special cases are $\mathcal P_i^{j}(i_1)=\mathcal M_{i_{1}}^{j}\times \mathcal M_i^{i_1}$ for $k=1$ and
 $\mathcal P_i^{j}(\emptyset)=\mathcal M_i^j$ for $k=0$. For $k\geq 1$ we compute, using multiplicative
 property of weights of composed strands, 
 \begin{equation}\label{eq-pathcalc}
 \begin{split} 
  \sum_{P\in \mathcal P_i^{j}(i_1\ldots i_k)}w(P) \quad &= 
  \sum_{P_k\in \mathcal M_{i_{k}}^{j}, \ldots, P_1\in 
   \mathcal M_{i_1}^{i_2}, P_0\in \mathcal M_i^{i_1}}w(P_k)\ldots w(P_1)w(P_0)\;\\
  &=
  \Bigl(\sum_{P_k\in \mathcal M_{i_k}^{j}}w(P_k)\Bigr)\ldots 
  \Bigl(\sum_{P_1\in \mathcal M_{i_1}^{i_2}} w(P_1)\Bigr)
   \Bigl(\sum_{P_0\in \mathcal M_i^{i_1}} w(P_0)\Bigr)\\
\mbox{\footnotesize by Lemma~\ref{lm-basicsLTW} part (2)\qquad}  
 &= \Burep n (b)_{ji_k}\ldots \Burep n (b)_{i_2i_1}\Burep n (b)_{i_1i}   \\
\mbox{\footnotesize using block form  \qquad}   &= Y_{ji_k}\ldots Q_{i_2i_1}Z_{i_1i}  \;. \\
 \end{split}
 \end{equation}
The set of admissible paths  $\mathcal P_i^{j}[k]$ from $(i,0)$ to $(j,1)$
that intersect $\mathcal L$ exactly $k$ times is the union of all 
$\mathcal P_i^{j}(i_1\ldots i_k)$ for fixed $i$, $j$, and $k\,$. Hence we obtain
for $k\geq 1$ from (\ref{eq-pathcalc}) by summation over all intermediate indices
$i_1,\ldots, i_k$ that
\begin{equation}\label{eq-C1}
 \sum_{P\in \mathcal P_i^{j}[k]}w(P)\,=\,(YQ^{k-1}Z)_{ji}\;.
\end{equation}
In the case $k=0$ we have $\mathcal P_i^{j}[0]=\mathcal{M}^j_i$ so that
\begin{equation}\label{eq-C2}
 \sum_{P\in \mathcal P_i^{j}[0]}w(P)\,=\,X_{ji}\;.
\end{equation} 
Summing terms in (\ref{eq-C1}) and (\ref{eq-C2}) we thus obtain
\begin{equation}\label{eq-Qgeomser}
 \sum_{P\in \mathcal P_i^{j}: \; |P\cap\mathcal L|\leq N}w(P)\,=\,
 \Bigl(X + Y\Bigl(\sum_{r=0}^{N-1}Q^{r}\Bigr)Z\Bigr)_{ji}\;.
\end{equation}
Given that $T$ is a  string link we know by Lemma~\ref{lm-Qnil} that $Q$ specializes to a 
nilpotent matrix at $t=1$ so that we can apply Lemma~\ref{lm-Qest}. Thus, if we choose
any $d<1$ in the latter lemma, the geometric series in (\ref{eq-Qgeomser})
will converge as $N\to\infty$  for $t$ in some $\varepsilon$-vicinity of 1 to the right hand side
expression of (\ref{eq-blfrm}). 
This has to coincide with the rational function limit asserted in Theorem~A of \cite{LTW98}.
By uniqueness of meromorphic continuations we thus
have the desired equality as rational functions for all $t$.
\end{proof}

We conclude with an observations related to the initial interpretation of
$\LTWstr n (T)$ as a stochastic matrix (at least for positive string links), namely
that there are right and left eigenvectors independent of $T$. To this end we
denote the $n$-dimensional row vector $\eige_n=(1,\ldots, 1)$ and let $\eigv_n$ be 
the $n$-dimensional column vector with $\eigv_n^T=(1,t,\ldots, t^{n-1})\,$.
\begin{corr} \label{corr-LTWeigen}
For any string link $T\in \str n $ we have 
\begin{equation}\label{eq-LTWeigen}
 \eige_n\LTWstr n (T)=\eige_n \qquad \mbox{ and } \qquad  \LTWstr n (T)\eigv_n=\eigv_n\;.
\end{equation}
\end{corr}
\begin{proof}
We first note that (\ref{eq-LTWeigen}) holds for braids since $\eige_n$ and $\eigv_n$
are easily verified to be eigenvectors for the braid generators in (\ref{eq-BnPres}). In particular,
for a braid $b\in B_{n+m}$ whose $m$-closure is $T$ we have 
$\eige_{n+m}\Burep n (b)=\eige_{n+m}$ and $\Burep n (b)\eigv_{n+m}=\eigv_{n+m}$. Since 
$\eige_{n+m}=(\eige_n,\eige_m)$ the former implies $\eige_nX+\eige_mZ=\eige_n$ and 
$\eige_nY+\eige_mQ=\eige_m$, which can also be written as $\eige_nY=\eige_m(\onem - Q)$
Using  (\ref{eq-blfrm}) we find 
$\eige_n\LTWstr n (T)=\eige_nX+\eige_nY(\onem -Q)^{-1}Z=\eige_nX+\eige_m(\onem - Q)(\onem -Q)^{-1}Z=
\eige_nX+\eige_m Z=\eige_n$. A similar calculation, using again Proposition~\ref{prop-blfrm},
shows that $\eigv_n$ is a right eigenvector.
\end{proof}

The fact that $\eige_n$ is a left eigenvector means that all column sums of $\LTWstr n (T)$ 
are one, supplementing a formal proof 
to the intuition for this fact provided in \cite{LTW98}. 
Assuming that $T$ is a  positive string link and $t\in[0,1]$ so that all entries 
in $\LTWstr n (T)$ are non-negative, we thus have that 
$\LTWstr n (T)$ is indeed a stochastic matrix. After suitable renormalization
the positive eigenvector $\eigv_n$  thus represents an equilibrium state 
$\eqp_n=\frac 1 {\langle \eige_n, \eigv_n\rangle}\eigv_n=(p_1,\ldots,p_n)^T$ with
probability of finding a ball in $j$-th position given by 
$p_j=\frac {1-t\,}{\,1-t^n}\, t^j$. If $T$ is in addition non-separable 
and $t\in (0,1)$ this
is the unique equilibrium, see \cite{LTW98}.

Since the stationary state $\eqp$ is independent of $T$ it clearly contains
no information about $T$, answering Remark (4) of \cite{LTW98} in the negative. An interpretation
of these right and left eigenvectors in terms of the representation theory of
$U_{-1}(\mathfrak{sl}_2)$ is be given in Section~\ref{ss-comm}.

\section{Tangle Functors and Exterior Algebras}\label{s-Rmat}

\subsection{Tangle Functor Associated to Alexander Polynomial}\label{ss-TFAP}

In this section we review, with slight variations, Ohtsuki's description in \cite{OH02}  of the tangle functor in
(\ref{eq-OFstrDec}) which is associated to the 
Alexander polynomial and generalizing the Burau representation. The construction is based in Tuarev's set of relations
for R-matrices identified in \cite{Tu89} which imply the extension to a functor on oriented tangles, as stated in Theorem~3.7 of \cite{OH02}. 
The R-matrix associated in  \cite{OH02} to the Alexander polynomial is given as an endomorphism on
$V\otimes V$ where $V$ is the free $\mathbb Z[t^{\frac 12},t^{-\frac 12 }]$ module with generators $e_0$ and $e_1$.
In the basis $\{e_0\otimes e_0,e_0\otimes e_1,e_1\otimes e_0,e_1\otimes e_1\}$ it is has the form
\begin{equation}\label{eq-BurRmat}
 R=\left[
\begin{array}{cccc}
 t^{-\frac 12} &0&0&0\\
0 & 0 & 1 &0\\
0 & 1 & t^{-\frac 12}-t^{\frac 12} & 0\\
0 & 0 & 0 & -t^{\frac 12 }\\
\end{array}
\right]\;.
\end{equation}
This R-matrix implements  a representation  of $\psi_n:B_n\to \mathrm{End}(V^{\otimes n})$ 
in the usual manner, see (2.1) of  \cite{OH02}, 
 so that  $\psi_n(b)$ coincides with  $\OhtsFct(b)$ for a braid $b\in B_n$. 

We note that our convention for orientations is the opposite of that in \cite{OH02}, where downwards arrows
are considered positive orientations. However, diagrams in \cite{OH02} there are easily translated to our 

As already indicated in Section~\ref{ss-functgl} of the introduction, the tangle functors 
preserves a natural grading on the associated vector spaces which can be expressed more 
formally as follows. Specifically, define an endomorphism $\ndgop(\underline \epsilon)$ 
on the module 
$\OhtsFct(\underline{\epsilon})=V^{\epsilon_1}\otimes\ldots \otimes V^{\epsilon_n}$
acting diagonally in the natural basis by
\begin{equation}\label{eq-gradop}
 \ndgop(\underline \epsilon)(e_{i_1}^{\epsilon_1}\otimes \ldots \otimes e_{i_n}^{\epsilon_n})
=\bigl(\sum_{s=1}^n\epsilon_si_s\bigr)e_{i_1}^{\epsilon_1}\otimes \ldots \otimes e_{i_n}^{\epsilon_n}\;.
\end{equation}
Here we use the convention $e_j^+=e_j$ and  $e_j^-=e_j^*$ for basis vectors of $V^+=V$ and $V^-=V^*$
respectively. The eigenspaces of $\ndgop(\underline \epsilon)$ are thus the graded components of 
$\OhtsFct(\underline \epsilon)$.  It is readily checked that the morphisms in (\ref{eq-gradop}) in fact
a natural transformation $\ndgop:\OhtsFct\stackrel{\bullet}{\longrightarrow}\OhtsFct$, which is 
another, equivalent way of saying that $\OhtsFct$ preserves the natural grading. 
convention by simply reversing all arrows. 

In the case when all signs are positive we denote further 
$\dgop^{\otimes}_n:=\ndgop(\posar{n})\in \mathrm{End}(V^{\otimes n})$ which has 
eigenvalue $k=|\{s:i_s=1\}|$ on  $e_{i_1}\otimes\ldots\otimes e_{i_n}\,$.
The eigenspace 
\begin{equation}\label{eq-grVdef}
 \grV n k=\mathrm{ker}(\dgop^{\otimes}_n-k\onem)
\end{equation}
is thus the $k$-graded 
component of rank $n \choose k$ in $V^{\otimes n}$, as noted in (\ref{eq-OFstrDec}), and is invariant under 
the braid group action. 

Observe also that the R-matrix in (\ref{eq-BurRmat}) specializes for $t=1$ to a signed permutation given by
$R(e_i\otimes e_j)=(-1)^{ij}e_j\otimes e_i$ so that in particular $R=R^{-1}$. The latter implies that in
this case crossings of a tangle $T$ can be changed arbitrarily without changing $\OhtsFct(T)$ so that we may
assume $T$ to be a braid. Hence $\OhtsFct(T)$ is a composition of signed permutations.
Moreover, on $\grV n 0=\langle e_0^{\otimes n}\rangle$ all $R$ acts as identity at $t=1$, which implies
the following statement.
\begin{lemma}\label{lm-t=1OF}
The endomorphism $\OFstr{n}{k}(T)$ reduce to signed permutations on the canonical basis in the specialization $t=1$.

In particular $\OFstr{n}{0}(T)=1$ at $t=1$. 
\end{lemma}

Another useful property of the tangle functor is its
equivariant with respect to the $U_{-1}(\mathfrak{sl}_2)$
action given in \cite{OH02}. In the particular case of a  tangle $T:\posar{n}\to \posar{n}$ this 
property implies 
that $\OhtsFct(T)\in\mathrm{End}(V^{\otimes n})$ commutes with operators
\begin{equation}\label{eq-htens}
 h^{\otimes n}= t^{\frac n2 }(-1)^{\dgop^{\otimes}_n}\;, \qquad \mbox{ where } \quad h=t^{\frac 12 }\left[\begin{array}{cc}
                                                  1 & 0 \\ 0 & -1 \\
                                                 \end{array}\right] \;,
\end{equation}
as well as 
\begin{equation}\label{eq-EFtens}
\begin{split}
 \widetilde E_n=&\sum_{i=1}^n\onem^{\otimes i-1}\otimes\left[\begin{array}{cc}
                                                  0 & 1 \\ 0 & 0 \\
                                                 \end{array}\right] \otimes (h^{-1})^{\otimes n-i}\\
\mbox{and}
  \qquad \widetilde F_n=&\sum_{i=1}^nh^{\otimes i-1}\otimes\left[\begin{array}{cc}
                                                  0 & 0 \\ 1 & 0 \\
                                                 \end{array}\right] \otimes  \onem^{\otimes n-i}\;.
\end{split}
\end{equation}
The operators in (\ref{eq-htens}) and (\ref{eq-EFtens}) describe the 
 actions of $n$-fold coproducts of rescaled generators of $U_{-1}(\mathfrak{sl}_2)$ and they fulfill
basic relations. For example, $h^{\otimes n}$ anti-commutes with both  $\widetilde E_n$ and $\widetilde F_n$,
and $[\widetilde E_n,\widetilde F_n]=  (t^{\frac 12}-t^{-\frac 12})^{-1}\bigl(h^{\otimes n}-(h^{-1})^{\otimes n}\bigr)\,$.

Functoriality also implies that the operator invariant of a tangle $T:\posar{n}\to\posar{n}$ is given by 
the partial quantum trace over the invariant for a braid $b\in B_{n+m}$ if $T$ is given as the closure of $b$
in the sense of Figure~\ref{fig:Markov_Pres}. 
More precisely,  for an endomorphism 
$f\in\mathrm{End}(V^{\otimes n})$ define its quantum trace in terms of the canonical trace as follows
\begin{equation}
 \qtr_n(f)= trace_{V^{\otimes n}}(h^{\otimes n}f)\;.
\end{equation}
It follows from the evaluations and coevaluations associated to extrema in Section~3.3 of \cite{OH02}
(again with reversed orientation convention) that closing off a right most strand of a  tangle diagram
with an arc corresponds to applying $\qtr_1$ to contract the respective indices of the associated operator.
Iterating the process we obtain for the closure $T$ of a braid $b$ as above
\begin{equation}\label{eq-OhClTr}
 \OhtsFct(T)\,=\,id^{\otimes n}\otimes \qtr_m(\OhtsFct(b))\,,
\end{equation}
where we suppress notation for the natural isomorphism 
$\mathrm{End}(V^{\otimes n+m})\cong\mathrm{End}(V^{\otimes n})\otimes\mathrm{End}(V^{\otimes n})\,$.

\subsection{Isomorphisms with Exterior  Algebra Representations}\label{ss-ExtAlg}

In this section we outline, again with some variations in conventions and normalizations, the equivalence of
the braid representation $\psi_n$ obtained from $\OhtsFct$ and the exterior algebra extension of the 
Burau representation given in Appendix~3 of \cite{OH02}.

Denote by $\Burmod_n$ the free $\mathbb Z[t^{\frac 12},t^{-\frac 12}]$-module with basis $\{v_1,\ldots, v_n\}$ 
and by $\Burep n : B_n\to \mathrm{End}(\Burmod_n)$ 
 the braid group representation as
given in 
(\ref{eq-burauR}) and (\ref{eq-burauM}). The extension to the exterior algebra is thus
\begin{equation}\label{eq-WdgBur}
\extalg * \Burep n:\,B_n\,\to\,\mathrm{End}\bigl(\extalg * \Burmod_n\bigr)\,:\;b\,\mapsto\,\extalg *\Burep n (b)\;.
\end{equation}
The action is clearly also graded with invariant submodules $\extalg{k}\Burmod_n$. As before  it is useful to
encode the grading as an operator on $\extalg{*}M_n$ defined by 
\begin{equation}
\dgop^{\wedge}_n(\omega)=k\omega\qquad \mbox{ for } \quad \omega\in \extalg{k}\Burmod_n\,.
\end{equation}
This allows us to define, similar to the quantum trace above, a supertrace $\sutr_n$
on morhisms $f\in\mathrm{End}\bigl(\extalg * \Burmod_n\bigr)$ by
\begin{equation}\label{eq-defsutr}
 \sutr_n(f) \, = \, trace_{\sextalg{*}\Burmod_n}\bigl((-1)^{\dgop^{\wedge}_n}f\bigr)\;.
\end{equation}

Analogous to (C.3) in Appendix~3 of \cite{OH02} we next define for each $n$ an isomorphism $\Isom_n:V^{\otimes n}\to \extalg{*}\Burmod_n$
by induction as follows
\begin{equation}\label{eq-TWisom}
\begin{tikzpicture}[scale=1]
 \node (A) at (0,0) {$\Isom_n:\,V^{\otimes n}$};
 \node (B) at (4.6,0) {$\extalg{*}M_{n-1}\otimes V$};
 \node (C) at (9,0) {$\extalg{*}M_{n}$};
 \node (E) at (4.5,-1) {$\alpha_0\otimes e_0 \,+\, \alpha_1\otimes e_1\;\,$};
 \node (F) at (9.3,-1) {$\;\,\alpha_0 \,+\, t^{\frac n 2}\alpha_1\wedge v_n\;.$};
 \node (G) at (-1,-1) {with};
 \path[font=\scriptsize, arrows={-angle 90}]
 (B) edge   (C)
  (A) edge node[above]{$\Isom_{n-1}\otimes id_V$} (B);
\path[font=\scriptsize, arrows={|-angle 90}]
   (E) edge   (F);
\end{tikzpicture}
\end{equation}
Some immediate properties of these isomorphisms implied by (\ref{eq-TWisom}) include
that they preserve the respective gradings, that is,
\begin{equation}\label{eq-isograd}
 \Isom_n\dgop^{\otimes}_n\,=\,\dgop^{\wedge}_n\Isom_n\;,
\end{equation}
and that they factor, up to scaling, with respect to products of spaces in the sense of the following commutative 
diagram:
\begin{equation}\label{eq-IsoFact}
 \begin{tikzpicture}[scale=1]
  \node (A) at (1.8,0) {$V^{\otimes n+m}\;\; $};
  \node (B) at (6.2,0) {$\;\;V^{\otimes n}\otimes V^{\otimes m}$};
  \node (C) at (0,1.7) {$\extalg{*}M_{n+m}\;\;$};
   \node (D) at (3.7,1.7) {$\;\extalg{*}(M_{n}\oplus M_m)\;$};
  \node (E) at (8,1.7) {$\;\;\extalg{*}M_{n}\otimes \extalg{*}M_{m}$};
  \path[-]
  ([yshift=1pt]C.east) edge  ([yshift=1pt]D.west)
  ([yshift=-1pt]C.east) edge  ([yshift=-1pt]D.west)
 ([yshift=1pt]D.east) edge  ([yshift=1pt]E.west)
  ([yshift=-1pt]D.east) edge  ([yshift=-1pt]E.west)
([yshift=1pt]A.east) edge  ([yshift=1pt]B.west)
 ([yshift=-1pt]A.east) edge  ([yshift=-1pt]B.west);
 \path[arrows={-angle 90 }]
 (A) edge node[left]{$\Isom_{n+m}$} (C)
  (B) edge node[right]{$\quad  \Isom_{n}\otimes t^{\frac n2\dgop^{\wedge}_m}\Isom_m$} (E);
\end{tikzpicture}
\end{equation}
Here double lines indicate obvious canonical isomorphisms and the factor $t^{\frac n2\dgop^{\wedge}_m}$ 
stems from the shift in the basis labeling for $M_{n+m}\cong M_n\oplus M_m$.

In order to state the equivalence of braid group representations we also consider the 
representation $\widehat\psi_n:B_n\to \mathrm{End}(V^{\otimes n})$ obtained from the 
rescaled R-matrix $\widehat R= t^{\frac 12}R$, with $R$ as in (\ref{eq-BurRmat}). 
It is related to the original representation by 
\begin{equation}\label{eq-psiresc}
 \widehat\psi_n(b)\,=\,t^{\frac 12 \swri (b)}\psi_n(b)\,=\,t^{\frac 12 \swri (b)}\OhtsFct(b)\;,
\end{equation}
where $\wri(b)$ is the writhe of $b$ given by the number of positive crossing minus the 
number of negative crossings of $b\,$.

The following lemma is essentially identical to Lemma~C.1 in \cite{OH02} and is verified by direct computation using 
(\ref{eq-TWisom}) and (\ref{eq-IsoFact}):
\begin{lemma}[{\cite{OH02}}] \label{lm-xBRequiv}
The isomorphisms defined in (\ref{eq-TWisom}) provides
and equivalence between the 
representations $\psi_n$ and $\extalg{*}\Burep n $ of $B_n$.
 That is, for any $b\in B_n$ we have the following commutative diagram. 
 \begin{equation}\label{eq-xBRequiv}
 \begin{tikzpicture}[scale=1]
  \node (A) at (0,0) {$V^{\otimes n}$};
  \node (B) at  (0,1.7) {$V^{\otimes n}$};
 \node (C) at (3,0) {$\extalg{*}M_n$};
 \node (D) at (3,1.7) {$\extalg{*}M_n$};
  \path[arrows={-angle 90 }]
   (B) edge node[above] {\small $\Isom_n$}(D)
   (A) edge node[below] {\small $\Isom_n$}(C)
   (B) edge node[left]{$\widehat\psi_n(b)\;$} (A)
    (D) edge node[right]{$\;\extalg{*}\Burep n(b)$} (C);
\end{tikzpicture}
 \end{equation}
Moreover, all isomorphisms preserve gradings so that  (\ref{eq-xBRequiv})
also holds when restricted to sub-representations $\grV nk$ and $\extalg{k}M_n$
instead of $V^{\otimes n}$ and $\extalg{*}M_n$.
\end{lemma}

\subsection{Relations Between Traces}\label{ss-TrRel}
The aim of this section is to replace the quantum trace formula for braid closures (\ref{eq-OhClTr}) by
traces over exterior algebras and thus reduce the proof of  Theorem~\ref{thm-main} to the exterior 
linear algebra of the Burau representation. We begin with notation for the partial supertrace given by
the following composite of natural isomorphisms and $\sutr_n$ as in (\ref{eq-defsutr}).
\begin{equation} \label{eq-defSTR}
 \begin{tikzpicture}[scale=1]
   \node (A) at (0,0) {$ \STR^{n+m}_n:\mathrm{End}\bigl(\extalg{*}M_{n+m}\bigr)$};
  \node (B) at (5.1,0) {$\mathrm{End}\bigl(\extalg{*}M_{n}\otimes\extalg{*}M_{m} \bigr)$};
  \node (C) at (8,0) {};
 \node (D) at (-1.4,-1.3) {};
 \node (E) at (2,-1.3){$\mathrm{End}\bigl(\extalg{*}M_{n}\bigr)\otimes\mathrm{End}\bigl(\extalg{*}M_{m}\bigr)\;$};
\node (F) at (9,-1.3){$ \;\mathrm{End}\bigl(\extalg{*}M_{n}\bigr)$}; 
 \path[-]
  ([yshift=1pt]B.east) edge  ([yshift=1pt]C.west)
  ([yshift=-1pt]B.east) edge  ([yshift=-1pt]C.west)
 ([yshift=1pt]D.east) edge  ([yshift=1pt]E.west)
  ([yshift=-1pt]D.east) edge  ([yshift=-1pt]E.west)
([yshift=1pt]A.east) edge  ([yshift=1pt]B.west)
 ([yshift=-1pt]A.east) edge  ([yshift=-1pt]B.west);
 \path[arrows={-angle 90 }]
(E) edge node[above] {$id\otimes \sutr_m$} (F);
  \end{tikzpicture}
\end{equation}
Moreover, for an endomorphism $\,f\in\mathrm{End}\bigr(\extalg{*}M_n\bigl)\,$ we denote the conjugate
\begin{equation}
 f^{\Isom}\,=\,\Isom_n^{-1}f\Isom_n\quad \in \;\mathrm{End}(V^{\otimes n})\;.
\end{equation}
The explicit relation between supertrace and quantum trace is  stated  in this terminology in the next lemma. 
\begin{lemma}\label{lm-traces}
 Suppose $f\in \mathrm{End}\bigl(\extalg{*}M_{n+m}\bigr)$ then
\begin{equation}\label{eq-traces}
 id\otimes \qtr_m(f^{\Isom})\,=\,t^{\frac m2} \STR^{m+n}_n(f)^{\Isom}\;.
\end{equation}
\end{lemma}
\begin{proof}
 By linearity we may assume that $f=a\otimes b$ with $a\in \mathrm{End}\bigr(\extalg{*}M_n\bigl)$ and
$b\in \mathrm{End}\bigr(\extalg{*}M_m\bigl)$  modulo the isomorphism indicated in (\ref{eq-defSTR}).
Using (\ref{eq-IsoFact}) we thus find
$$
f^{\Isom}\,=\,\Isom_{n+m}^{-1}f\Isom_{n+m}\,=\,\Isom_n^{-1}a\Isom_n\otimes \Isom_m^{-1}t^{-\frac n2\dgop^{\wedge}_m}b t^{\frac n2\dgop^{\wedge}_m}\Isom_m
$$
so that we find for the partial quantum trace
\begin{equation}
\begin{split} 
 id\otimes \qtr_m(f^{\Isom})\,
&=\,\Isom_n^{-1}a\Isom_n \cdot trace_{V^{\otimes m}}\Bigl(h^{\otimes m}\Isom_m^{-1}t^{-\frac n2\dgop^{\wedge}_m}b t^{\frac n2\dgop^{\wedge}_m}\Isom_m\Bigr) \\
&=\,a^{\Isom} \cdot trace_{\sextalg{*}M_m}\bigl(Pb\bigr)
\end{split} 
\end{equation}
using cylicity of the canonical trace in the last step to combine appearing isomorphisms in the isomorphism $P$ which is given and 
evaluated as follows.
\begin{equation}
\begin{split}
  P&=t^{\frac n2\dgop^{\wedge}_m}\Isom_m h^{\otimes m}\Isom_m^{-1}t^{-\frac n2\dgop^{\wedge}_m}
=^{by (\ref{eq-htens})}t^{\frac m2}t^{\frac n2\dgop^{\wedge}_m}\Isom_m(-1)^{\dgop_m^{\otimes}}\Isom_m^{-1}t^{-\frac n2\dgop^{\wedge}_m}\\
&=^{by (\ref{eq-isograd})}t^{\frac m2}t^{\frac n2\dgop^{\wedge}_m} (-1)^{\dgop_m^{\wedge}} t^{-\frac n2\dgop^{\wedge}_m}\,=\,
t^{\frac m2}  (-1)^{\dgop_m^{\wedge}}\;.
\end{split}
\end{equation}
On the other hand we have $\STR^{n+m}_n(f)=a\cdot\sutr_m(b)=a\cdot trace_{\extalg{*}M_m}\bigr( (-1)^{\dgop_m^{\wedge}}b\bigr)$
by definitions in (\ref{eq-defsutr}) and (\ref{eq-defSTR}) from which (\ref{eq-traces}) readily follows.
\end{proof}

From the traces equivalence we can now compute the tangle functor $\OhtsFct$ on string links from
the Burau representation.
\begin{corr} \label{corr-OhtStrRel}
Suppose $T:\posar{n}\to\posar{n}$ is a string link presented as the closure of a braid $b\in B_{n+m}$.
Then
\begin{equation}\label{eq-OhtStrRel}
 \OhtsFct(T)= t^{\frac 12 (m+\swri(b))}\cdot \STR^{n+m}_n(\extalg{*}\Burep n (b))^{\Isom}\;.
\end{equation} 
\end{corr}

\begin{proof} The proof is a direct computation combining previous results
\begin{equation}
 \begin{split}
\mbox{\small Right hand side of (\ref{eq-OhtStrRel})\quad } &=  t^{\frac 12 \swri(b)}t^{\frac m2}\cdot \STR^{n+m}_n(\extalg{*}\Burep n (b))^{\Isom}\\\
\mbox{\small by Lemma~\ref{lm-traces}\quad }   &= t^{\frac 12 \swri(b)}id\otimes \qtr_m\Bigl(\extalg{*}\Burep n (b)^{\Isom}\Bigr)\\
\mbox{\small by Lemma~\ref{lm-xBRequiv}\quad }   &= t^{\frac 12 \swri(b)}id\otimes \qtr_m\bigl(\widehat\psi_{n+m}(b)\bigr)\\
\mbox{\small by (\ref{eq-psiresc})\quad}
&= id\otimes \qtr_m\bigl(\psi_{n+m}(b)\bigr)= id\otimes \qtr_m\bigl(\OhtsFct(b)\bigr)\\
\mbox{\small by (\ref{eq-OhClTr})\quad} & =\OhtsFct(T)\,.\\
 \end{split} 
\end{equation}
\end{proof}

Note that Corollary~\ref{corr-OhtStrRel} not only implies that the maps $\STR^{n+m}_n(\extalg{*}\Burep n (b))$  
preserve the natural grading but also that these operators commute with the actions  of the conjugates
$\breve E_n=\Isom_n\widetilde E_n\Isom_n^{-1}$ and $\breve F_n=\Isom_n\widetilde F_n\Isom_n^{-1}$ of the operators in 
(\ref{eq-EFtens}). In fact, similar to Lemma~C.3 in \cite{OH02} we have 
\begin{equation}\label{eq-Fbreve}
 \breve F_n\alpha=t^{-\frac 12}\eigv_n\wedge \alpha\;
\end{equation}
for any $\alpha\in\extalg{*}M_n$ and $\eigv_n$ as in Corollary~\ref{corr-LTWeigen}. 
In order the find a respective expression for $\breve E_n$
we define
$\covol{j}:\extalg{k}M_n\to\extalg{k-1}M_n$ by 
$\covol{j}(\alpha)=0$ and  $\covol{j}(\alpha\wedge v_j)=\alpha$
where  $\alpha=v_{i_1}\wedge\ldots\wedge v_{i_s}$
 with $j\not\in\{i_1,\ldots,i_s\}$. A basic calculation with forms then yields
\begin{equation}\label{eq-Ebreve}
 \breve E_n\,=\, t^{-\frac n2}\sum_{j=1}^n\covol{j}\;.
\end{equation}

For a string link $T$ given as the 
closure of a braid $b$ we denote the restriction of the operator 
to degree $k$ forms as
\begin{equation}\label{eq-defBRT}
 \BRT n k (T,b)\,=\,\STR^{n+m}_n(\extalg{*}\Burep n (b))\Big|_{\extalg{k}M_n}\;.
\end{equation}
This corresponds, via conjugation by $\Isom_n$, to the restrictions defined in (\ref{eq-wdefLTW}) so that we have from
(\ref{eq-OhtStrRel}) that
 $\OFstr n k(T)= t^{\frac 12 (m+\swri(b))}\cdot\BRT n k (T,b)^{\Isom}\,$. Let us also define
 \begin{equation}
\BRT n {k/0} (T,b)=\frac 1{\BRT n 0 (T,b)}\BRT n k (T,b)\;.
\end{equation}
 Canceling the factors we thus find 
that 
\begin{equation}
 \OFstr n {k/0}(T)=\BRT n {k/0} (T,b)^{\Isom}\;
\end{equation}
where the morphism on the left is as  defined in (\ref{eq-def10quot}).
Consequently we have  reduced the proof of Theorem~\ref{thm-main} to showing that 
\begin{equation}\label{eq-redrel}
\extalg{k}\LTWstr n (T) =\BRT n {k/0} (T,b) \qquad \mbox{ for all \ \ } T\,, 
\end{equation}
which will be the objective of the next section.

\section{Proof of Main Results and Conclusions}\label{s-proof}

Before proving Theorem~\ref{thm-main} in Section~\ref{ss-proof} we provide several technical 
lemmas on exterior algebras that relate partial traces, actions on top forms 
and their dual contractions, as well as Schur complements. At the end of this section we 
comment on various implications of our result and possible generalizations. 

\subsection{Supertrace from Top Forms}\label{ss-strviatop} The aim of this section is to generalize the well known
relation $\,\sutr_m(\sextalg{*}f)=\det(\onem-f)\,$, where $f\in\mathrm{End}(M_m)$ is any endomorphism on
a  free module $M_m$ of rank $m$, to partial traces with respect to endomorphisms on $U\oplus M_m$
where $U$ is another module of rank $n$.

Let $M_m$ be generated by a basis $\{w_1,\ldots,w_m\}$ and denote by $\subs m$ the set of subsets
of $\{1,\ldots , m\}$. For any  $S=\{i_1,\ldots,i_k\}\in\subs m$ with $i_1<i_2<\ldots<i_k$ denote
$\volel S=w_{i_1}\wedge\ldots\wedge w_{i_k}\,$ so that $\{\volel{S}:\,S\in \subs m\}$ is a basis
of $\extalg{*}M_m\,$.  Denote also the top form $\volel{m}=\volel{\{1,\ldots,m\}}= w_1\wedge\ldots\wedge w_m$.  We define contractions with respective dual forms on the combined exterior algebra
as 
\begin{equation}\label{eq-contr}
\begin{split}
  \covol S:\extalg{*}(U\oplus M_m)\to \extalg{*}U \quad &\mbox{with \ }
 \covol S(\alpha\wedge \volel{T})\,=\,\delta_{S,T}\alpha\\
& \mbox{ for all \ } \alpha\in \extalg * U, \;\;S,T\in\subs m\;.
\end{split}
\end{equation}

For future application we also record here  the following elementary property, that is immediate from (\ref{eq-contr}).
\begin{equation}\label{eq-covolfact}
 \covol{S}(\gamma\wedge\delta)=\gamma\wedge\covol{S}(\delta) \qquad \mbox{ for any \ } \gamma\in \extalg * U \mbox{ \ and \ } \delta\in\extalg{*}(U\oplus M_m)\,.
\end{equation}

We use again abbreviated notation $\covol{m}=\covol{\{1,\ldots,m\}}$ for contractions with
the respective top form. 
Using contractions the partial supertrace for an endomorphism $f\in\mathrm{End}(\extalg *( U\oplus M_m))$ acting on an
element $\alpha\in\extalg{*}U$ may thus be reexpressed by the following formula.
\begin{equation}\label{eq-strme}
 \STR^{n+m}_n(f)\alpha \;=\;\sum_{S\in \subs{m}} (-1)^{|S|}\covol{S}\bigl( f(\alpha\wedge\volel{S})\bigr)\;.
\end{equation}

Writing $S^c$ for the complement of $S$
we have the relation
\begin{equation}\label{eq-volScompS}
 \volel m=\sigma_S\volel S\wedge\volel {S^c}
\end{equation}
where $\sigma_S\in\{\pm1\}$ is the signature of the respective shuffle permutation. 
More generally, we have for $\beta\in\extalg{*}U$ that $\beta\wedge\volel T\wedge\volel {S^c}$ is 
a non-zero multiple of $\beta\wedge \volel{m}$ only if $T=S$. This observation and (\ref{eq-volScompS})
thus imply the  relation 
\begin{equation}\label{eq-covolScompS}
 \covol{S}(\eta)\,=\,\sigma_S\covol{m}(\eta\wedge\volel{S^c})\;.
\end{equation}

 \begin{lemma}\label{lm-str=top}
Let $A\in\mathrm{End}(U\oplus M_m)$ and $\alpha\in\extalg{*}U\,$. Then
\begin{equation}\label{eq-str=top}
 \STR^{n+m}_n(\extalg{*}A)\alpha \,=\,\covol{m}\Bigl(  \bigl(\extalg{*}A\alpha\bigr)\wedge \bigl(\extalg{*}(\onem-A)\volel{m}\bigr)\Bigr)\,.
\end{equation}
 \end{lemma}
\begin{proof} We first compute the action of $(\onem-A)$ on the top form.
 \begin{equation}\label{eq-1mAcomp}
\begin{split}
  \extalg{*}(\onem-A)\volel{m}&=(\onem-A)w_1\wedge\ldots \wedge (\onem-A)w_m\\
&=\sum_{\epsilon_1,\ldots,\epsilon_n\in\{0,1\}}(-1)^{\sum_i\epsilon_i}A^{\epsilon_1}w_1\wedge\ldots \wedge A^{\epsilon_n}w_n\\
&=\sum_{S\in \subs m}(-1)^{|S|}A_S\tau_m\quad =\quad \sum_{S\in \subs m}\sigma_S(-1)^{|S|} A_S(\volel S\wedge\volel {S^c})\\
&= \sum_{S\in \subs m}\sigma_S(-1)^{|S|}(\extalg {|S|}A\volel S)\wedge\volel {S^c}
\end{split}
 \end{equation}
Here $A_S$ acts on $\volel{m}$ by the formula in the previous line with $\epsilon_j=1$ if $j\in S$ and $\epsilon_j=0$ otherwise.
Moreover, we are making use of (\ref{eq-volScompS}) in the third  line. We conclude the proof by further evaluation.
\begin{equation}
\begin{split}
\mbox{\small Right hand side of (\ref{eq-str=top})\quad}
&=
\sum_{S\in \subs m}\sigma_S(-1)^{|S|} \covol{m}\Bigl(  \bigl(\extalg{*}A\alpha\bigr)\wedge \bigl(\extalg {|S|}A\volel S)\wedge\volel {S^c}\Bigr)
\\
&=
\sum_{S\in \subs m}\sigma_S(-1)^{|S|} \covol{m}\Bigl(  \bigl(\extalg{*}A(\alpha \wedge  \volel S)\bigr)\wedge\volel {S^c}\Bigr)
\\
\mbox{\small by (\ref{eq-covolScompS})\quad}&=
\sum_{S\in \subs m} (-1)^{|S|} \covol{S}\bigl(\extalg{*}A(\alpha \wedge  \volel S) \bigr) 
\\
\mbox{\small by (\ref{eq-strme})\quad}&=
\STR^{n+m}_n\bigl(\extalg{*}A)\alpha \,.
\\
\end{split}
\end{equation}
\end{proof}

\subsection{Partial Traces from Schur Complements}\label{ss-trschur} Consider an endomorphism 
 $B\in \mathrm{End}(U\oplus M_m)$ and assume that for the respective block form  
\begin{equation}\label{eq-Bblockf}
 B=\left[\begin{array}{cc}
          H & J\\
        K & L
         \end{array}\right]  \qquad  L\in \mathrm{End}(M_m) \mbox{ \ is  invertible}.
\end{equation}
Then $B$ has a block-UL factorization  
\begin{equation}\label{eq-RLdec}
 B=B_uB_l
\qquad \mbox{with}\qquad 
B_u=\left[\begin{array}{cc}
          \onem & G\\
        0 & \onem
         \end{array}\right]
\mbox{\ and \ }
B_l= \left[\begin{array}{cc}
          D & 0\\
        K & L
         \end{array}\right] 
\end{equation}
where
\begin{equation}\label{eq-schurd}
 D=H-JL^{-1}K \qquad \mbox{and} \qquad G=JL^{-1}\,.
\end{equation}

The endomorphism $D$ is also called the {\em Schur complement} of $L$ in $B$. A basic relation 
between these endomorphisms is $\det(B)=\det(L)\det(D)$ which we generalize in the next lemma for our purposes.

\begin{lemma} \label{lm-extschurd}
Suppose $B$ has a block form as in (\ref{eq-Bblockf}), $L$ invertible, and $D\in\mathrm{End}(U)$ its Schur complement as in
(\ref{eq-schurd}). Then for $\,\alpha\in\extalg{*}U\,$ we have 
\begin{equation}\label{eq-extschurd}
 \covol{m}\bigl(\extalg{*}B(\alpha\wedge\volel{m})\bigr)\,=\,\det(L)\extalg{*}D\alpha\,.
\end{equation} 
\end{lemma}

\begin{proof} We note first that, since $B_l$ maps $M_m$ to itself and since its restriction to $M_m$ is $L$, we have
$\extalg{*}B_l\volel{m}=\extalg{*}L\volel{m}=\det(L)\volel{m}$. Furthermore, for $\alpha=x_1\wedge\ldots\wedge x_k$ for some 
$x_j\in U$ 
we have $\extalg{*}B_l\alpha =(Dx_1+Kx_1)\wedge\ldots\wedge (Dx_k+Kx_k) = \extalg{*}D\alpha+\rho$ where 
$\rho=\sum_i\gamma_i\wedge w_i$ since each $Kx_i\in M_m$ and is hence a combination of $w_i$'s. This form then implies
$
 \bigl(\extalg{*}B_l\alpha\bigr)\wedge\volel{m}=\bigl(\extalg{*}D\alpha\bigr)\wedge\volel{m}\,
$. Combining these formulas for $\extalg{*}B_l$ we thus obtain
\begin{equation}\label{eq-actBl}
\begin{split}
  \extalg{*}B_l(\alpha\wedge\volel{m})&=(\extalg{*}B_l\alpha)\wedge (\extalg{*}B_l\volel{m})
=\det(L)(\extalg{*}B_l\alpha)\wedge  \volel{m} \\
&= \det(L)(\extalg{*}D\alpha)\wedge\volel{m} \;.\\
\end{split}
\end{equation}
Continuing with the action of $B_u$ on forms, we note that  $\extalg{*}B_u\eta=\eta$ for any $\eta\in\extalg{*}U$ 
since  $B_u$ is identity on $U$. Furthermore we have that 
$\extalg{*}B_u\volel{m}=(w_1+Gw_1)\wedge\ldots\wedge  (w_n+Gw_n)=\volel{m}+\sum_{S\in\subs m: |S|<m}\psi_S\wedge \volel{S}\,$
where $\psi_S\in\extalg{*}U$ since all  $Gw_i\in U$.  As a result we have
$$
\extalg{*}B_u(\eta\wedge \volel{m})= \eta\wedge \extalg{*}B_u\volel{m}
= \eta\wedge \extalg{*}B_u\volel{m}=\eta\wedge \volel{m}+\rho'
$$
where $\rho'$ is the summation of terms $\eta\wedge\psi_S\wedge\volel{S}$ with $|S|<m$ and $\eta\wedge\psi_S\in \extalg{*}U$.
Since all of these terms are by (\ref{eq-contr}) in the kernel of $\covol{m}$ we find
\begin{equation}\label{eq-actBu}
 \covol{m}(\extalg{*}B_u(\eta\wedge \volel{m}))=\eta\;.
\end{equation}
Combining the actions in (\ref{eq-actBl}) and  (\ref{eq-actBu}) for $\eta=\det(L)\extalg{*}D\alpha$, we thus find
 \begin{equation}
  \begin{split}
   \covol{m}\bigl(\extalg{*}B_uB_l(\alpha\wedge\volel{m})\bigr)
&=\covol{m}\Bigl(\extalg{*}B_u\bigl(\extalg{*}B_l(\alpha\wedge\volel{m})\bigr)\Bigr)\\
&=\covol{m}\Bigl(\extalg{*}B_u\bigl(\det(L)(\extalg{*}D\alpha)\bigr)\wedge\volel{m}))\Bigr)\\
&=\det(L)(\extalg{*}D\alpha)\,\\
  \end{split} 
 \end{equation}
which is the desired form and thus completes the proof.
\end{proof}

The next lemma extends this result to expressions as those in Lemma~\ref{lm-str=top}.

\begin{lemma} \label{lm-extschurd2}
Let $B$, $L$, $D$, and $\alpha$ be as in Lemma~\ref{lm-extschurd} above. Then
\begin{equation} \label{eq-extschurd2}
 \covol{m}\Bigl(\bigl(\extalg{*}(\onem-B)\alpha\bigr)\wedge(\extalg{*}B\volel{m})\Bigr)\,=\,\det(L)\extalg{*}(\onem-D)\alpha\,.
\end{equation}
\end{lemma}

\begin{proof} The calculations for this proof are very similar to those of Lemma~\ref{lm-str=top}. We may assume 
$\alpha=x_1\wedge\ldots\wedge x_k$ for independent generators $x_i\in U$. As before denote
$\alpha_S=x_{i_1}\wedge\ldots\wedge x_{i_p}$ for $S=\{i_1,\ldots,i_p\}\in\subs k$ with $i_1<\ldots < i_p$ so that
\begin{equation}
 \alpha=\sigma'_S\alpha_{S^c}\wedge\alpha_{S}\;
\end{equation}
where $\sigma'_S$ is the  signature of the respective shuffle permutation. We thus obtain by a calculation analogous to that in
(\ref{eq-1mAcomp}) that
\begin{equation}\label{eq-Bfact}
 \extalg{*}(\onem-B)\alpha\,=\,\sum_{S\in\subs k}\sigma'_S(-1)^{|S|}\alpha_{S^c}\wedge\bigl(\extalg{*}B\alpha_{S}\bigr)\;.
\end{equation} 
Forming the wedge product of this expression with $\extalg{*}B\volel{m}$ we thus obtain  
$$
\bigl(\extalg{*}(\onem-B)\alpha\bigr)\wedge(\extalg{*}B\volel{m})\,=\,\sum_{S\in\subs k}\sigma'_S(-1)^{|S|}\alpha_{S^c}\wedge\bigl(\extalg{*}B(\alpha_{S}\wedge\volel{m})\bigr)\;.
$$
The remainder of the proof is a computation.
\begin{equation}
 \begin{split}
  \mbox{\small Left hand side of (\ref{eq-extschurd2})\quad}& = 
        \sum_{S\in\subs k}\sigma'_S(-1)^{|S|}\covol{m}\Bigl(\alpha_{S^c}\wedge\bigl(\extalg{*}B(\alpha_{S}\wedge\volel{m})\bigr)\Bigr)\\
   \mbox{by (\ref{eq-covolfact})\quad}& = 
        \sum_{S\in\subs k}\sigma'_S(-1)^{|S|} \alpha_{S^c}\wedge\covol{m}\bigl(\extalg{*}B(\alpha_{S}\wedge\volel{m})\bigr)\\
  \mbox{by Lemma~\ref{lm-extschurd}\quad}& = 
        \det(L)\sum_{S\in\subs k}\sigma'_S(-1)^{|S|} \alpha_{S^c}\wedge  \extalg{*}D \alpha_{S} \\
 \mbox{ \quad}& = 
         \det(L)\extalg{*}(\onem-D) \alpha\ \\
\end{split}
\end{equation}
where the last step is analogous to (\ref{eq-Bfact}).
\end{proof}

Substituting $B=\onem-A$ we combine this with Lemma~\ref{lm-str=top} to obtain the following.

\begin{corr}\label{corr-extform} 
 Suppose an endomorphism $A\in \mathrm{End}(U\oplus M_m)$ is of the form
 \begin{equation}\label{eq-Acond}
  A=\onem_{n+m} - \left[\begin{array}{cc}
          \onem_n & G\\
        0 & \onem_m
         \end{array}\right] \left[\begin{array}{cc}
          D & 0\\
        K & L
         \end{array}\right] 
 \end{equation}
with $L$ invertible. Then we have
\begin{equation}\label{eq-extform}
 \STR^{n+m}_n(\extalg{*}A) = \det(L)\extalg{*}(\onem-D)\;.
\end{equation}
\end{corr}

\subsection{Proof of Main Results and Example}\label{ss-proof}
\begin{proof}[Proof of Theorem~\ref{thm-main}]  
As noted in Section~\ref{ss-TrRel} it suffices to prove the relation in (\ref{eq-redrel}). 
For a string link $T:\posar{n}\to\posar{n}$ presented by the closure braid $b\in B_{n+m}$
consider the Burau matrix $\Burep n (b)$ with block form as in (\ref{eq-bclmat}). 

In order to evaluate $\STR^{n+m}_n(\extalg{*}\Burep n (b))$ used in the definition of $\BRT{n}{k}(T,b)$
we apply Corollary~\ref{corr-extform} using $A=\Burep n (b)$ and $U=M_n$.  The condition in (\ref{eq-Acond})
then translates to the set of block conditions $Z=-K$, $Q=\onem-L$, $Y=-GL$, and $X=\onem-(D+GK)=\onem-H$.

We note that by Proposition~\ref{prop-blfrm} the block $L=\onem-Q$ is indeed invertible. Moreover,
$\onem-D=X+GK=X-GZ=X+YL^{-1}Z=X+Y(\onem-Q)^{-1}Z=\LTWstr n (T)$, also by  Proposition~\ref{prop-blfrm}.
This implies by (\ref{eq-extform}) that
\begin{equation}
 \STR^{n+m}_n(\extalg{*}\Burep n (b))=\det(\onem-Q)\extalg{*}\LTWstr n (T)\;.
\end{equation}
From (\ref{eq-defBRT}) we thus find by restriction to $k$-forms that
\begin{equation}
 \BRT n k (T,b)=\det(\onem-Q)\extalg{k}\LTWstr n (T)\;.
\end{equation}
so that in particular $\BRT n 0 (T,b)=\det(\onem-Q)$. Relation (\ref{eq-redrel}) now
follows immediately completing the proof of Theorem~\ref{thm-main}. 
\end{proof}

\begin{proof}[Proof of Corollary~\ref{corr-main}]  
The $k\times k$ minors of $\OFstr{n}{1}$ are just the matrix elements of $\extalg{k}\OFstr{n}{1}(T)$
which is by (\ref{eq-01spcase}), up to relabeling and rescaling of basis, the same as 
$\extalg{k}\bigl(\OFstr{n}{0}(T)\LTWstr{n}(T)\bigr)=\OFstr{n}{0}(T)^k\extalg{k}\LTWstr{n}(T)$.
At the same time $\OFstr{n}{0}(T)\extalg{k}\LTWstr{n}(T)$ is, again up to relabeling and 
rescaling of basis, the same as $\OFstr{n}{k}(T)$ by Theorem~\ref{thm-main}. Hence, up to
permutations and multiplications by units in $\Rh$, the matrix elements of $\extalg{k}\OFstr{n}{1}(T)$ 
are the same as those of  $\OFstr{n}{0}(T)^{k-1}\OFstr{n}{k}(T)$. Since $\OFstr{n}{k}(T)$ is a 
matrix over $\Rh$ this implies the assertion.
\end{proof}

As an example consider again the string link $S:\posar{2}\to\posar{2}$ from Figure~\ref{fig:ex_strlk}.
The Ohtsuki's tangle functor can be readily computed, for example, by using skein relations
$R-R^{-1}=(t^{-\frac 12}-t^{\frac 12})\onem$ and the fact that isolated components render a diagram zero.
Organized by graded components and using the basis $\{e_0\otimes e_1, e_1\otimes e_0\}$ for $\grV{2}{1}$ we
obtain
\begin{equation}\label{eq-VonS}
\begin{split} 
 \OhtsFct(S)&=\OFstr 2 0 (S) \oplus \OFstr 2 1 (S) \oplus \OFstr 2 2 (S) \\
& = (2-t) \,\oplus\,
 \left[\begin{array}{cc}
         3-t-\overline t & t^{-\frac 12}- t^{\frac 12}\\
         t^{-\frac 12}- t^{\frac 12} & 1\\
       \end{array}\right]
       \,\oplus \, (2-\overline t)\;.
\end{split}
\end{equation}
We find from  (\ref{eq-TWisom}) that $\Isom_2:\grV{2}{1}\to \extalg{1}M_2=M_2$ is given by
$\Isom_2(e_0\otimes e_1)=tv_2$ and $\Isom_2(e_1\otimes e_0)=t^{\frac 12}v_1$ so that in the 
basis $\{v_1,v_2\}$ for $M_2$ we obtain
\begin{equation}
 \Isom_2\OhtsFct(S)\Isom_2^{-1}= (2-t) \,\oplus\,
 \left[\begin{array}{cc}
         1 & \overline t -1 \\
       1 - t  & 3-t-\overline t\\
       \end{array}\right]
       \,\oplus \, (2-\overline t)\;.
\end{equation}
Given that $\OFstr{2}{0}(S)=(2-t)$ we immediately have that $\OFstr{2}{1/0}(S)$ is, up to basis change, 
the same as $\extalg{1}\LTWstr{2}(S)=\LTWstr{2}(S)$ as in (\ref{eq-LTWstr-exS}). Moreover, we readily compute
\begin{equation}
 \extalg{2}\LTWstr{2}(S)=\det(\LTWstr{2}(S))=\frac {2-\overline t}{2-t} = \OFstr{2}{2/0}(S)\;,
\end{equation}
thus verifying the statement of Theorem~\ref{thm-main} for all $k$  in this example.

\subsection{Concluding Comments and Outlook}\label{ss-comm}\

We begin with remarks on the probabilistic motivation initially given in \cite{Jo87} and expanded upon in
\cite{LTW98}. As noted earlier to have a true stochastic matrix  we need to confine ourselves to diagrams 
$T$ with only positive crossings and $t\in[0,1]$ in order to probabilities in $[0,1]$ (or all negative crossings and $t^{-1}\in(0,1]$).

Recall from the comments following Corollary~\ref{corr-LTWeigen} that 
the unique  equilibrium state  $\eqp_n=c\eigv_n$ of
$\LTWstr{n}(T)$ for non-separable $T$ is  fixed and thus contains no topological information about $T$. 
The size of the space of equilibrium states for general $T$, however, may be used to 
 provide a measure for of separability of a given string link and thus may be of interest for further study.

The existence of the fixed equilibrium state and the interpretation of $\LTWstr{n}(T)$ as a stochastic matrix
rely on the existence of left and right  eigenvectors with eigenvalue 1 as in Corollary~\ref{corr-LTWeigen}.
From the point of view of the tangle functor their existence is actually a basic consequence of the underlying 
representation theory of $U_{-1}(\mathfrak{sl}_2)$ as outline in the following. 

Particularly,  we have, by equivariance and accounting for degrees, for the operators defined in (\ref{eq-EFtens})
that 
$\widetilde E_n\OFstr{n}{k+1}(T)=\OFstr{n}{k}(T)\widetilde E_n$ and 
 $\widetilde F_n\OFstr{n}{k}(T)=\OFstr{n}{k+1}(T)\widetilde F_n$. 
Let now $\mathbf{1}=e_0^{\otimes n}$ be the generating vector of $\grV n 0$ and $\mathbf{1}^*$ the respective dual
vector in $\grV n 0^*$. The intertwining relations imply for $k=0$ that 
$\OFstr{n}{1}(T)(\widetilde F_n\mathbf{1})=\widetilde F_n\OFstr{n}{0}(T)\mathbf{1}=\OFstr{n}{0}(T)(\widetilde F_n\mathbf{1})$
so that $\OFstr{n}{1/0}(T)(\widetilde F_n\mathbf{1})=(\widetilde F_n\mathbf{1})$. From Theorem~\ref{thm-main} we thus
have that $\Isom_n\widetilde F_n\mathbf{1}\,$ is also an eigenvector with eigenvalue one for $\LTWstr{n}(T)$. 

Using  (\ref{eq-Fbreve}) we find that this is by $\Isom_n\widetilde F_n\mathbf{1}=\breve F_n\Isom_n\mathbf{1}=\breve F_n\Isom_n\mathbf{1}
=\breve F_n1=t^{-\frac 12}\eigv_n$ indeed proportional to the eigenvector found directly in Corollary~\ref{corr-LTWeigen}. 
A similar equivariance argument shows that a left eigenvector for $\LTWstr{n}(T)$ is given by 
$\mathbf{1}^*\widetilde E_n\Isom_n^{-1}=\mathbf{1}^*\Isom_n^{-1}\breve E_n= \breve E_n\big|_{M_n}=t^{-\frac n2}\eige_n\,$
by (\ref{eq-Ebreve}). We remark also that the {\em reduced } Burau representation can be viewed in this context as
the further restriction of $M_n$ to the ``highest'' weight space $\mathrm{ker}(\widetilde E_n)$. 

Another question not treated here but of possible further interest is that of duality relations between degree $k$ and 
$(n-k)$ representations. The existence of a duality principle is strongly suggested by the  obvious the symmetry 
in the example in (\ref{eq-VonS}) but also basic algebraic properties of the R-matrix used for $\OhtsFct$ and its 
relations to exterior algebras. From a topological point of view such a principle may be expected via Poincar\'e 
duality in local coefficient cohomologies of underlying configuration spaces.

Observe also that the multiplicative function $\OFstr{n}{0}:\str{n}\to\Rh$ allows us to define
the assignment 
\begin{equation}
 \mathcal X: \str{n}\to\mathbb Z^{0,+}\,:\;T\,\mapsto\,\mathrm{span}\bigl(\OFstr{n}{0}(T)\bigr).
\end{equation}
where the span of a Laurent polynomial is given by the difference of its highest and lowest
non-vanishing power in $t^{\frac 12}$. Clearly, this is  an additive function with  $\mathcal X(T\circ S)= 
\mathcal X(T)+ \mathcal X(S)$ and vanishes on braids, that is, the  invertible elements of $\str{n}$.

A further objects of study is thus the size pf the kernel of $\mathcal X$ beyond the braid group, as well as relations
of $\mathcal X(T)$ with the number simple loops in the random walk picture or other topological
properties of string links. 

We finally point out the generalization of Ohtsuki's functor indicated by the two-variable R-matrix
provided at the end of Section~4.5 in \cite{OH02}. This yields, in the same manner as before, a 
tangle functor from which we obtain, analogous to (\ref{eq-wdefLTW}),   a representation
\begin{equation}\label{eq-pStrFun}
\pOFstr n j :\; \pstr n \,\to\,\mathrm{End}\bigl(\mathbb{Z}[t_1^{\pm\frac 14},\ldots, t_n^{\pm\frac 14}]^{{n\choose j}}\bigr)
\end{equation}
on the monoid $\pstr{n}$ of {\em pure} string links. In the tangle functor picture the variables $t_j$ label
representations of  $U_{-1}(\mathfrak{sl}_2)$ and thus fit into the framework of TQFTs where the $t_j$
are interpreted as ``colors'' or ``charges''. Analogous to (\ref{eq-def10quot}) we can also define 
quotient representations $\pOFstr{n}{j/0}$. 
 
Extending the construction of the classical Gassner representation of the pure braid groups 
Kirk, Livingston, and Wang obtain in \cite{KLW01} a representation 
$$\gamma^{KLW}_n: \pstr{n}\to \mathrm{End}\bigl(\mathbb{Q}(t_1,\ldots, t_n)^{n}\bigr)\,.
$$
We thus conclude this article with a  conjecture that naturally extends Theorem~\ref{thm-main}
and for which we expect the proof  to follow a similar strategy.
\begin{conj}
 The representation $\gamma^{KLW}_n$ is equivalent to $\,\pOFstr{n}{1/0}$. 
\end{conj}


\end{document}